\documentclass{amsart}
\usepackage{tabu}
\usepackage{amssymb} 
\usepackage{amsmath} 
\usepackage{amscd}
\usepackage{amsbsy}
\usepackage{comment, enumerate}
\usepackage[matrix,arrow]{xy}
\usepackage{hyperref}
\usepackage{mathrsfs}
\usepackage{color}
\usepackage{mathtools,caption}
\usepackage{todonotes}

\newcommand{\dis}{\displaystyle}

\newcommand{\QQ}{\mathbb Q}
\newcommand{\ZZ}{\mathbb Z}

\newcommand{\Fn}{\mathbb F_n}
\newcommand{\Fl}{\mathbb{F}_\ell}

\newcommand{\OL}{\mathcal O_L}
\newcommand{\rs}{\sqrt{-7}}
\newcommand{\bpi}{\bar\pi}
\newcommand{\bgamma}{\bar\gamma}
\newcommand{\plambda}{\lambda^\prime}
\newcommand{\pE}{E^\prime}
\newcommand{\pEo}{E_1^\prime}
\newcommand{\pEfo}{E_{f^\prime_1}}
\newcommand{\pfo}{f^\prime_1}
\newcommand{\fp}{\mathfrak p}

\newcommand{\fn}{\mathfrak n}

\newcommand{\pNo}{N^\prime_1}
\newcommand{\pRo}{R_1^\prime}
\newcommand{\rhoffn}{\rho_{f,\fn}}
\newcommand{\rhoE}{\rho_{E,n}}
\newcommand{\rhof}{\rho_{f,n}}
\newcommand{\rhoEo}{\rho_{E_1,n}}
\newcommand{\rhopEo}{\rho_{\pEo,n}}
\newcommand{\rhoEt}{\rho_{E_2,n}}

\newcommand{\rhofofn}{\rho_{f_1,\fn}}
\newcommand{\rhofo}{\rho_{f_1,n}}
\newcommand{\rhopfofn}{\rho_{f^\prime_1,\fn}}
\newcommand{\rhopfo}{\rho_{f^\prime_1,n}}
\newcommand{\rhoftfn}{\rho_{f_2,\fn}}
\newcommand{\rhoft}{\rho_{f_2,n}}

\DeclareMathOperator{\ord}{ord}
\DeclareMathOperator{\Gal}{Gal}
\DeclareMathOperator{\Rad}{Rad}
\DeclareMathOperator{\res}{res}
\DeclareMathOperator{\norm}{Norm}
\DeclareMathOperator{\GL}{GL}
\DeclareMathOperator{\Tr}{Tr}
\DeclareMathOperator{\Frob}{Frob}

\newtheorem{theorem}{Theorem}[section]
\newtheorem{lemma}[theorem]{Lemma}
\newtheorem{proposition}[theorem]{Proposition}

\begin{document}

\title[An application of the modular method]{An application of the modular method and the symplectic argument to a Lebesgue--Nagell equation}

\author{Angelos Koutsianas}
\address{Department of Mathematics, The University of British Columbia, 1984 Mathematics Road, Vancouver, BC, V6T 1Z2, Canada}
\email{akoutsianas@math.ubc.ca}

\date{\today}


\begin{abstract}
In this paper, we study the generalized Lebesgue--Nagell equation
$$
x^2+7^{2k+1}=y^n.
$$
This is the last case of equations of the form $x^2+q^{2k+1}=y^n$ with $k\geq0$ and $q>0$ where $\QQ(\sqrt{-q})$ has class number one. Our proof is based on the modular method and the symplectic argument.
\end{abstract}

\maketitle

\section{Introduction}

One of the Diophantine equation which has a very rich history is the Lebesgue-Nagell equation
\begin{equation}\label{eq:LebesgueNagell}
x^2+D=y^n,\qquad x,y,D,n\in\ZZ, n\geq 3.
\end{equation}
The first who solved \eqref{eq:LebesgueNagell} for general $n$ is Lebesgue for $D=1$ \cite{Lebesgue50} while Nagell solves the cases $D=3,5$ \cite{Nagell48} in $1923$. The cases $0<D\leq 100$ were studied by many, Cohn solves $77$ of the cases \cite{Cohn93}, Mignotte-De Weger solve the missing by Cohn cases for $D=74,86$ \cite{MignotteWeger96} and Bennett-Skinner solve the cases $D=55,95$ \cite{BennettSkinner04}. Finally, the remaining $19$ were solved by Bugeuad-Mignotte-Siksek \cite{BugeaudMignotteSiksek06} using a variety of modern techniques as linear of forms in logarithms and modular method. 

All the above cases were for fix $D$. In recent years, many authors have studied \eqref{eq:LebesgueNagell} in the more general case when the prime factors of $D$ lie in a finite set $S$. In other words, the solutions of the equation
\begin{equation}
x^2+p_1^{a_1}\cdots p_k^{a_k}=y^n,\qquad x,y,D,a_i,n\in\ZZ,~n\geq 3,~a_i\geq 0.
\end{equation}
for a fix set $S=\{p_1,\cdots,p_k\}$. The case when the size of $S$ is small has gained the interest of several authors. For example, for $S=\{2\}$ the equation was studied by Brown \cite{Brown77}, Arief-Muriefah \cite{AriefMuriefah97}, Cohn  \cite{Cohn92} and Le \cite{Le02}. The case $S=\{3\}$ is studied in \cite{Brown77,ArifMuriefah98,Luca00,Tao08} and the case $S=\{5\}$ in \cite{ArifMuriefah99,Muriefah06,Tao09}. In \cite{BerczesPink08} B\'{e}rczes-Pink completely solve \eqref{eq:LebesgueNagell} for $S=\{q\}$ with $2\leq q\leq 100$, $k$ even and $\gcd(x,y)=1$. The literature is very rich for sets $S$ with size bigger than $1$ and we refer to \cite{BerczesPink12,BerczesPink08,CangulDemirciInamLucaSoydan13} for a more detailed exposition of the known results.

It is well-known that $\QQ(\sqrt{-q})$ has class number one for $q = 1$, $2$, $3$, $7$, $11$, $19$, $43$, $67$ and $163$. It is a natural question to study \eqref{eq:LebesgueNagell} for $D=q^k$ and $q$ as above. As we have already mentioned the cases $q=1,2,3$ have completely solved thanks to the contribution of many authors \cite{Lebesgue50,Brown77,Cohn92,AriefMuriefah97,ArifMuriefah98,Luca00,Le02,Tao08}. Due to Arif-Muriefah \cite{ArifMuriefah02}, Saradha-Srinivasan \cite{SaradhaSrinivasan06},  Zhu-Le \cite{ZhuLe11} and B\'{e}rczes-Pink \cite{BerczesPink08} the solutions for the cases $q=11,19,43,67,163$ have also been classified. For $D=7^k$  Luca-Togb\'{e} completely solve \eqref{eq:LebesgueNagell} for $k$ even \cite{LucaTogbe07} and the case $k$ odd and $y$ odd is studied in \cite{SaradhaSrinivasan06}.

In this paper, we study the case $q=7$ and $k$ odd. From the above we understand that it is enough to consider $y$ being even because this is the remaining case of \eqref{eq:LebesgueNagell} for $D=7^k$. A solution of \eqref{eq:LebesgueNagell} is called \textit{non-trivial} if $x\neq 0$. The main result of the paper is the following.

\begin{theorem}\label{thm:main}
All non-trivial solutions of the Lebesgue-Nagell equation
\begin{equation}\label{eq:main}
x^2+7^{2k+1}=y^n,\qquad x,y,k,n\in\ZZ,~n\geq 3,~k\geq 0
\end{equation}
are
\begin{align*}
n = 3 & \qquad (|x|,y,k) = (7^{3\lambda},2\cdot 7^{2\lambda},3\lambda), (13\cdot 7^{3\lambda},8\cdot 7^{2\lambda},3\lambda+1),\\
 & \qquad (147\cdot 7^{3\lambda},28\cdot 7^{2\lambda},3\lambda+1),(1911\cdot 7^{3\lambda},154\cdot 7^{2\lambda},3\lambda+1),\\
 & \qquad (3\cdot 7^{3\lambda+2},4\cdot 7^{2\lambda+1},3\lambda+1),(7^{3\lambda+2},2\cdot 7^{2\lambda+1},3\lambda+1),\\
  & \qquad (39\cdot 7^{3\lambda+2},22\cdot 7^{2\lambda+1},3\lambda+1),\\
n = 4 & \qquad (|x|,y,k) = (3\cdot 7^{2\lambda},2\cdot 7^\lambda,2\lambda),\\
n = 5 & \qquad (|x|,y,k) =  (5\cdot 7^{5\lambda},2\cdot 7^{2\lambda},5\lambda),(181\cdot 7^{5\lambda},8\cdot 7^{2\lambda},5\lambda).
\end{align*}
for $\lambda\geq 0$ when $k,n$ satisfy one of the following conditions,
\begin{itemize}
\item$n=3,4,5$ or,
\item $n$ is divisible by a prime congruent to $13,17,19,23\pmod{24}$ and $k\not\equiv 1\pmod 3$ or,
\item $n$ is divisible by a prime congruent to $5,7,13,23\pmod{24}$ and $k\equiv 1\pmod 3$. 
\end{itemize}
\end{theorem}

From the results of Luca-Togb\'{e} \cite{LucaTogbe07}, Saradha-Srinivasan \cite{SaradhaSrinivasan06} and Theorem \ref{thm:main} we have the following theorem.

\begin{theorem}
Suppose $n$ is a prime such that $n\equiv 13,23\pmod{24}$. Then, the Lebesgue-Nagell equation
\begin{equation}
x^2+7^k=y^n,\qquad x,y,k\in\ZZ,~k\geq 0
\end{equation}
has no non-trivial solutions.
\end{theorem}

The proof of Theorem \ref{thm:main} is based on the modular method and the symplectic argument when $n$ is `big'. For a fix $n$ the solution of \eqref{eq:main} is reduced to the solution of Thue and Thue-Mahler equations and we solve these equations for the `small' $n$'s.

The paper is organised as follows. In Section \ref{sec:modular_method} we recall known results and terminology about the modular method and the symplectic argument. We also briefly explain the main idea of the symplectic argument. In Section \ref{sec:preliminaries} we show that a non-trivial solution of \eqref{eq:main} corresponds to a non-trivial primitive solution of one of the equations $x^2+7^{2k+1}=y^n$ or $7x^2+1=y^n$. We also show that the problem can be reduced to the solution of Thue and Thue-Mahler equations.  At the end of section we solve \ref{eq:main} for the special case $y=2$ which we need in the modular method. In Section \ref{sec:case_a} we study the non-trivial primitive solutions of $x^2+7^{2k+1}=y^n$ when $y$ is even. We introduce two Frey curves of signature $(p,p,2)$ and $(2,3,p)$. We use the local information of Frey curves to prove the non-existence of solutions

The required computations for this paper have been done using \texttt{Sage} \cite{Sage} and \texttt{Pari} \url{http://pari.math.u-bordeaux.fr/}. The code can be found at author's homepage \url{https://sites.google.com/site/angeloskoutsianas/research}.

\section{Modular method}\label{sec:modular_method}

Our proof will make use of the modular method and the symplectic argument. We use Frey curves, modularity, Galois representations and level lowering. In this section we will recall some of the results and terminology that are used in the proof. The reader can find a more detailed exposition of the techniques and ideas in e.g. \cite[Chapter 15]{Cohen07}.

Suppose $f$ is a cuspidal newform of weight $2$ and level $N_f$ with $q$-expansion
$$
f(q) = q + \sum_{i\geq 2}a_i(f)q^i.
$$
We denote by $K_f$ the \textit{eigenvalue field} of $f$ and we say $f$ is \textit{irrational} if $[K_f:\QQ]>1$ and \textit{rational} otherwise. Suppose $n$ is a rational prime and $\fn$ a prime ideal in $K_f$ above $n$. Then, we can associate a continuous, semisimple representation 
$$\rhoffn:~\Gal(\overline{\QQ}/\QQ)\rightarrow \GL_2(\overline{\mathbb F}_n)$$
that is unramified at all primes away from $2n$ and $\Tr(\rhoffn(\Frob_n)\equiv a_n(f)\pmod{\fn}$.
If $f$ is rational then we use the notation $\rhof$ and we denote by $E_f$ a corresponding to $f$ elliptic curve over $\QQ$ due to modularity.

Suppose $E$ is an elliptic curve over $\QQ$ with conductor $N_E$. For a prime $\ell$ of good reduction of $E$ we have $a_\ell(E) = \ell + 1 - \#\tilde E(\Fl)$, where $\tilde E$ is the reduction of $E$ at $\ell$. We denote by $\rhoE$ the Galois representation of $\Gal(\overline{\QQ}/\QQ)$ on the $n$-torsion subgroup of $E$.

\begin{proposition}\label{prop:congruence_criterion}
With the above notation we assume $\rhoE\simeq\rhoffn$. Then there exists a prime ideal $\fn\mid n$ of $K$ such that for all rational primes $\ell$
\begin{enumerate}[(i)]
\item if $\ell\nmid n N_E N_f$ then $a_{\ell}(f)\equiv a_\ell(E)\pmod{\fn}$,

\item if $\ell\nmid nN_f$ and $\ell\parallel N_E$ then $(\ell + 1)\equiv\pm a_\ell(f)\pmod{\fn}$.
\end{enumerate}
\end{proposition}

The above proposition is standard (see \cite[p. 203]{Serre87}, \cite[Proposition 4.3]{BennettSkinner04}, etc) and we use it to bound $n$.

Suppose $E$ and $\pE$ are two elliptic over $\QQ$ and $n$ an odd prime. Let $\phi:E[n]\rightarrow\pE[n]$ be a $\Gal(\overline{\QQ}/\QQ)$-modules isomoprhism, hence there is an element element $r(\phi)\in\Fn^*$ such that, for all $P, Q\in E[n]$, the Weil pairings satisfy $e_{\pE,n}(\phi(P),\phi(Q))=e_{E,n}(P,Q)^{r(\phi)}$. We say that $\phi$ is  a \textit{symplectic isomorphism} (resp. \textit{anti-symplectic isomorphism}) if $r(\phi)$ is a square (resp. non-square) residue modulo $n$. If all $\Gal(\overline{\QQ}/\QQ)$-modules isomorphisms $\phi:E[n]\rightarrow\pE[n]$ are symplectic (resp. anti-symplectic) we say that the \textit{symplectic type} of the triple $(E,\pE,n)$ is symplectic (resp. anti-symplectic).

A variety of criteria that determine the symplectic type of $(E,\pE,n)$ using local information of $E,\pE$ have been developed even from 90's \cite{KrausOsterle92} while very recently we have the complete classification \cite{FreitasKraus17}. A combination of the modular method and symplectic criteria, which is called \textit{the symplectic argument}, has been used to the solution of Diophantine equations (\cite{Freitas16}, \cite{FreitasKraus16}, \cite{FreitasNaskreckiStoll17}, \cite{BennettBruniFreitas18}). The main idea is the following; after applying modularity and level lowering, we get an isomorphism $\rhoE\simeq\rhoffn$ for some newform of `suitable' level $N_f$ and the Frey curve $E$. Suppose $f$ is rational and $E_f$ the corresponding elliptic curve due to modularity. Suppose we have two primes $\ell_1,\ell_2$ for which we can determine the symplectic type $(E,E_f,n)$ using only local information of $E$ and $E_f$. In many cases the symplectic type given by $\ell_1$ and $\ell_2$ is different when $n$ satisfies specific congruence condition. This allows us to prove that no solutions arise from $E_f$ for these exponents $n$.

\section{Preliminaries}\label{sec:preliminaries}

In this section we show that a non-trivial solution of \eqref{eq:main} corresponds to a non-trivial primitive solution of one of the two equations
$$x^2+7^{2k+1}=y^n\qquad \text{or}\qquad 7x^2+1=y^n.$$
We say a solution \textit{primitive} if $\gcd(x,y)=1$. In addition, we show that for a fix exponent $n$ the problem of determining the non-trivial primitive solutions of the above two equations is reduced to the problem of solving Thue and Thue-Mahler equations. Finally, we solve \eqref{eq:main} for the case $y=2$.

The case $y$ is odd has been studied in \cite{SaradhaSrinivasan06}, so for the rest of the paper we make the assumption that $y$ is even. Moreover, we define $L=\QQ(\rs)$ and denote by $\OL$ the ring of integers of $L$. It holds that $2=\pi_2\bpi_2$ where $\pi_2=\frac{1+\rs}{2}$ and $\bpi_2=\frac{1-\rs}{2}$. It is also known that $\OL=\ZZ[\pi_2]$. We use the same notation for the prime ideals generated by $\pi_2$ and $\bpi_2$.


\begin{lemma}\label{lem:split_two_eq}
A non-trivial solution $(a,b)$ of \eqref{eq:main} corresponds to a non--trivial primitive solution to one of the following two equations
\begin{equation}\label{eq:main_a_lemma}
x^2+7^{2k+1}=y^n,
\end{equation}
or
\begin{equation}\label{eq:main_b_lemma}
7x^2+1=y^n,
\end{equation}
with $k\geq 0$ and $n\geq 3$.
\end{lemma}

\begin{proof}
Let $(a,b)$ be a solution of \eqref{eq:main}. Suppose $p$ is a prime such that $p\mid (a,b)$. Then from \eqref{eq:main} we understand that $p=7$. We write $a=7^ta_1$ and $b = 7^sb_1$ where $(a_1,7)=(b_1,7)=1$. So equation \eqref{eq:main} becomes
\begin{equation}
7^{2t}a_1^2 + 7^{2k+1} = 7^{ns}b_1^n.
\end{equation}
By looking to the exponents and keeping in mind that $(a_1,7)=(b_1,7)=1$, we have that either $2t=ns < 2k+1$ or $2k+1=ns<2t$. So, from the initial case we get
\begin{equation}
a_1^2+7^{2(k-t)+1}=b_1^n,
\end{equation}
which corresponds to a non-trivial primitive solution of \eqref{eq:main_a_lemma} while from the latter case we get
\begin{equation}
7\left(7^{t-k-1}a_1\right)^2+1=b_1^n,
\end{equation}
which corresponds to a non-trivial primitive solution of \eqref{eq:main_b_lemma}.
\end{proof}

\vspace*{0.3cm}
\subsection{Reduction to Thue and Thue-Mahler equation}\label{sec:Thue_ThueMahler}

Suppose $(a,b)$ is a non-trivial primitive solution of \eqref{eq:main_a_lemma}. We also assume that $n$ is an odd integer. Factorizing \eqref{eq:main_a_lemma} over $L$ we have
\begin{equation}\label{eq:fac_case_a}
(a + 7^k\rs)(a - 7^k\rs) = b^n.
\end{equation}
Let $\fp$ be a prime ideal of $\OL$ that divides $\gcd(a+7^k\rs,a-7^k\rs)$, this implies that $\fp\mid 2\rs$. However, $a,b,7$ are pairwise coprime and $b$ even, hence $\fp\mid 2$. Because the class number of $L$ is one and after possibly replacing $\gamma$ with $-\gamma$ below we have,
\begin{equation}
a+7^k\rs = \pi_2^{a_1}\bpi_2^{a_2}\gamma^n,
\end{equation}
where $\gamma\in\OL$ and $a_1,a_2 < n$ positive integers. After conjugating and adding or subtracting we have
\begin{align}
7^k & = \frac{\pi_2^{a_1}\bpi_2^{a_2}\gamma^n - \pi_2^{a_2}\bpi_2^{a_1}\bgamma^n}{2\rs},\\
a & = \frac{\pi_2^{a_1}\bpi_2^{a_2}\gamma^n + \pi_2^{a_2}\bpi_2^{a_1}\bgamma^n}{2}.
\end{align}
Because $a$ is odd, $2=\pi_2\bpi_2$ and $a_1,a_2<n$ we conclude that either $a_1\leq 1$ or $a_2\leq 1$. Moreover, from \eqref{eq:fac_case_a} we have that $a_1+a_2=n$.  Without loss of generality we assume that $a_2\leq 1$. Suppose $\gamma = u+v\pi_2$ where $u,v\in\ZZ$. Because $a,7$ are coprime and $n>1$ we undestand that $u,v$ are coprime. The same approach works when $n$ is even. To sum up, we have the following proposition.

\begin{proposition}\label{prop:ThueMahler}
Suppose $(a,b)$ is a non-trivial primitive solution of \eqref{eq:main_a_lemma}. Then, there exist coprime integers $u,v\in$ such that
\begin{equation}
a+7^k\rs = \pm \alpha^e(u+v\pi_2)^n,
\end{equation}
where $e=0,1$ and $\alpha = 2\pi_2^{n-2}$. The pair $(u,v)$ is an integral solution of the Thue-Mahler equation,
\begin{equation}
\pm 7^k = \frac{\alpha^e(u+v\pi_2)^n - \bar\alpha^e(u+v\bpi_2)^n}{2\rs}.
\end{equation}
\end{proposition}

Working similarly with the equation \eqref{eq:main_b_lemma} we have the following proposition,

\begin{proposition}\label{prop:Thue}
Suppose $(a,b)$ is a non-trivial primitive solution of \eqref{eq:main_b_lemma}. Then, there exist coprime integers $u,v$ such that
\begin{equation}
1+a\rs = \pm\alpha^e(u+v\pi_2)^n,
\end{equation}
where $e=0,1$ and $\alpha=2\pi_2^{n-2}$. The pair $(u,v)$ is an integral solutions of the Thue equation,
\begin{equation}
\pm 1 = \frac{\alpha^e(u+v\pi_2)^n + \bar\alpha^e(u+v\bpi_2)^n}{2}.
\end{equation}
\end{proposition}
 
Methods of solving Thue equations have been developed by Tzanakis and de Weger \cite{TzanakisWeger89} and Bilu and Hanrot and implemented in \texttt{Pari} and \texttt{Magma} \cite{Magma}. On the other hand, even though there are methods of solving Thue-Mahler equations, which have also been developed by Tzanakis and de Weger \cite{TzanakisWeger92}, there was no a general implementation to any of the known software packages. Very recently Adela Gherga has written \texttt{Magma} code implementing the method in \cite{TzanakisWeger92} for her thesis, together with some important improvements, and parts of this work will be included in \cite{GhergaKanelMatschkeSiksek}.

\vspace*{0.3cm}
\subsection{The case $y=2$}

Because of a technical step when we apply the modular method we have to exclude the case $y=2$. We do that in the following proposition.

\begin{proposition}\label{prop:b_equal_2}
The solutions of equation
\begin{equation}
x^2 + 7^{2k+1}=2^n,
\end{equation}
are $(|x|,k,n)=\{(1,0,3), (3, 0, 4), (5, 0, 5), (11, 0, 7), (13, 1, 9), (181, 0, 15)\}$.
\end{proposition}

\begin{proof}
Suppose $a$ is an odd integer which is a solution of the above equation for some $k$ and $n$,
$$a^2+7^{2k+1}=2^n.$$
Factorizing over $L$ the last equation we have
\begin{equation}
(a + 7^k\rs)(a - 7^k\rs) = 2^n.
\end{equation}
Hence, we understand that there exist $e_1,e_2\in\ZZ$ such that
\begin{equation}
a + (\rs)^{2k+1} = (\pm 1)\cdot\pi_2^{e_1}\cdot\bpi_2^{e_2}.
\end{equation}
We have to mention that $e_1 + e_2 = n$ and $e_1,e_2\geq 0$. Taking conjugate of the above equation and subtracting we end up with the equation
\begin{align*}
& 2\cdot (\rs)^{2k+1} = (\pm 1)\cdot\pi_2^{e_1}\cdot\bpi_2^{e_2}+ (\pm 1)\cdot\pi_2^{e_2}\cdot\bpi_2^{e_1} \Leftrightarrow\\
1 & = (\pm 1)\cdot\pi_2^{e_1-1}\cdot\bpi_2^{e_2-1}\cdot(\rs)^{-2k-1}+ (\pm 1)\cdot\pi_2^{e_2-1}\cdot \bpi_2^{e_1-1}\cdot (\rs)^{-2k-1}.
\end{align*}
The last equation is an $S$-unit equation. Based on the work of Smart \cite{Smart95,Smartbook}and Tzanakis and de Weger \cite{TzanakisWeger92,TzanakisWeger89} we can find an upper bound of $e_1$, $e_2$ and $2k+1$. Actually we have $e_1, e_2, 2k+1\leq 646$. Then, after an elementary computation we have the result.
\end{proof}

\section{Small exponents}

As we have seen in Section \ref{sec:Thue_ThueMahler} for a fix exponent $n$ the solution of equations \eqref{eq:main_a_lemma} and \eqref{eq:main_b_lemma} is reduced to the solution of Thue and Thue-Mahler equations. In this section we consider the cases $n=3,4,5$ and $7$. For $n=7$ we also assume $k\equiv 1\pmod 3$.

\begin{lemma}\label{lem:ThueMahler_3_4_5_7}
The non-trivial primitive solutions of \eqref{eq:main_a_lemma} for $n=3,4,5$ or $n=7$ with $k\equiv 1\pmod 3$ are 
\begin{align*}
(|x|,|y|,k,n) = & (1,2,0,3),(13,8,1,3),(147,28,1,3),(1911,154,1,3),\\
 & (3,2,0,4),(5,2,0,5),(181,8,0,5).
\end{align*}
\end{lemma}

\begin{proof}
For the case $n=7$ we have made the assumption $k\equiv 1\pmod 3$. In \cite{PoonenSchaeferStoll07} the authors compute the primitive solutions of the equation $x^2+y^3=z^7$. None of the solutions has $y$ a power of $7$, hence \eqref{eq:main_a_lemma} has no solutions for $n=7$ with $k\equiv 1\pmod 3$.

Suppose $n=5$. According to Proposition \ref{prop:ThueMahler} we have to solve the following Thue-Mahler equations
\begin{align}
2\cdot 7^k & = v(5u^4 + 10u^3v - 10u^2v^2 - 15uv^3 - v^4),\\
7^k & = -u^5 - 15u^4v - 10u^3v^2 + 50u^2v^3 + 35uv^4 - 3v^5.
\end{align}
where $u,v\in\ZZ$ are as in Proposition \ref{prop:ThueMahler}. We recall that $u,v$ are coprime. For the first equation we conclude that either $v=\pm 2$ and $5u^4 + 10u^3v - 10u^2v^2 - 15uv^3 - v^4 = \pm 7^k$ or $v=\pm 7^k$ and $5u^4 + 10u^3v - 10u^2v^2 - 15uv^3 - v^4=\pm 2$. Because of the even degree of the above polynomial we can assume that either $v=2$ or $v=7^k$ respectively. It is enough to consider the latter equation $\pmod 5$ to conclude that there are no solutions. For the case $v=2$ we have the solution $(u,k)=(-1,2)$. So, any effort to prove the non-existence of solutions will fail. However, we can bound $k$ reducing the problem to an $S$-unit equation where $k$ appears as an exponent of one of the generators. Because $v=2$ we actually have to solve the equation 
\begin{equation}
5u^4 + 20u^3 - 40u^2 - 120u - 16 = \pm 7^k.
\end{equation}
We define $g(u)=5u^4 + 20u^3 - 40u^2 - 120u - 16$ and $K = \QQ(\theta)$ where $\theta^4 + 20\theta^3 - 200\theta^2 - 3000\theta - 2000=0$. Using \texttt{Sage} we know that the extension $K/\QQ$ is Galois because the minimal polynomial of $\theta$ completely splits over $K$, $K$ has class number one, there is only one prime ideal $\fp_7$ above $7$ with ramification index $2$ and $g$ splits completely over $K$. We denote by $\rho_i$ the four roots of $g$ which all lie in $K$ and $\pi_7$ a uniformizer of $\fp_7$ which we fix. Then there exist units $\epsilon_i$ such that $u-\rho_i=\epsilon_i\pi_7^{k_i}$ where $k_i$ are non-negative integers with $k_1+k_2+k_3+k_4=2k$. Because $\Gal(K/\QQ)$ acts transitively over $\rho_i$'s and $u$ is an integer we understand that all $k_i$'s are the same and equal to $k/2$. Without loss of generality we have 
$$
u - \rho_1 = u - \theta/5=\epsilon_1\pi_7^{k/2}\qquad\text{and}\qquad u-\rho_2=u + \theta/5+2=\epsilon_2\pi_7^{k/2}
$$
Subtracting the last two equation we have,
\begin{equation}
1= \frac{\epsilon_2\pi_7^{k/2} - \epsilon_1\pi_7^{k/2}}{2(\theta/5+1)}. 
\end{equation}
The last equation equation is an $S$-unit equation and similarly to the proof of Proposition \ref{prop:b_equal_2} we prove that $k/2\leq 11157$. As long as we have bounded $k$ we can easily prove that $(u,v,k)=(\pm 1,\mp 2,2)$ are the only solutions. However, they correspond to the trivial solution with $a=0$. 

For the second Thue-Mahler equation $7^k = -u^5 - 15u^4v - 10u^3v^2 + 50u^2v^3 + 35uv^4 - 3v^5$ we asked Adela Gherga to solve the equation using her code and the solutions are the following,
\begin{align*}
(u,v,k) = & ~ ( 2, -1, 0 ), ( -1, 0, 0 ).
\end{align*}
From the solutions of the Thue-Mahler equation we have that the solutions of \eqref{eq:main_a_lemma} for $n=5$ are $(|a|,b)=(5,2),(181,8)$.

Suppose $n=4$ and $(a,b)$ is a non-trivial primitive solution of \eqref{eq:main_a_lemma}. It holds $a^2 + 7^{2k+1}=b^4$ which is rewritten by
\begin{equation}
(b^2 - a)(b^2 + a) = 7^{2k+1}.
\end{equation}
Because $a,b$ are coprime we understand that there exist positive integers $t,s$ such that $b^2-a=7^s$ and $b^2+a=7^t$ with $s+t=2k+1$. This implies that $2b^2=7^t+7^s$ and $2a=7^t-7^s$. Because $a,b$ are coprime we conclude that either $t=0$ or $s=0$, so $2b^2 = 7^{2k+1} + 1$. From \cite[Theorem 1.1]{BennettSkinner04} we have that there is no solution for $k>1$, thus we conclude there is only one solution $(|a|,|b|)=(3,2)$ for $k=0$.

Finally, we consider the case $n=3$. According to Proposition \ref{prop:ThueMahler} we have to solve the following Thue-Mahler equations
\begin{equation}
2\cdot 7^k = v (3u^2 + 3uv - v^2)\qquad \text{and}\qquad 7^k=u^3 + 3u^2v - 3uv^2 - 3v^3,
\end{equation}
where $u,v\in\ZZ$ are as in Proposition \ref{prop:ThueMahler}. We recall that $u,v$ are coprime. For the first equation we conclude that either $v=\pm 2$ and $3u^2 + 3uv - v^2 = \pm 7^k$ or $v=\pm 7^k$ and $3u^2 + 3uv - v^2=\pm 2$. Like the case $n=5$ we can assume that either $v=2$ or $v=7^k$ respectively. We can exclude the case $v=7^k$ by taking the equation either $\pmod 7$ or $\pmod 2$ and the case $v=2$ with the plus sign by taking the equation $\pmod 3$. Finally, for the case $v=2$ and $3u^2 + 3uv - v^2 = -7^k$ we have $3u^2+6u-4=3(u+1)^2-7=-7^k$ from which we understand that $k\leq 1$. For $k\leq 1$ there is the solution $(u,v,k)=(-1,2,1)$ which corresponds to the trivial solution $a=0$.

For the second Thue-Mahler equation $7^k=u^3 + 3u^2v - 3uv^2 - 3v^3$, we asked Adela Gherga to solve the equation using her code and the solutions are the following,
\begin{align*}
(u,v,k) = & ~(-4, 1, 1), (-1, 2, 1),(5, 4, 1), (4, -1, 1), (1, -2, 1), (-5, -4, 1), \\ 
& ~(2, -3, 0), ( -1, 0, 0), (-2, 1, 1), ( -2, 3, 0 ), ( 1, 0, 0 ), ( 2, -1, 1 ).
\end{align*}
From the solutions of the Thue-Mahler equation we have that the solutions of \eqref{eq:main_a_lemma} for $n=3$ are $(|a|,b)=~(1,2),(13,8),(147,28),(181,64),(1911,154)$ of \eqref{eq:main_a_lemma}.
\end{proof}

\begin{lemma}\label{lem:Thue_3_4_5_7}
The non-trivial primitive solutions of \eqref{eq:main_b_lemma} for $n=3,4,5,7$ are $(|x|,n)= (3,3),~(1,3),~(39,3)$.
\end{lemma}

\begin{proof}
According to Lemma \ref{prop:Thue} the solution of \eqref{eq:main_b_lemma} for fix $n$ is reduced to the solution of Thue equations. We have written a \texttt{Sage} program (based on functions by \texttt{Pari}) that solves the Thue equations and computes the solutions.
\end{proof}

With Lemmas \ref{lem:ThueMahler_3_4_5_7} and \ref{lem:Thue_3_4_5_7} we have determined the solutions in Theorem \ref{thm:main} for the `small' exponents. For the rest of paper we apply the modular method and the symplectic argument to prove the non-existence of solutions for the `big' exponents under the assumptions of Theorem \ref{thm:main}.

\section{Primitive solutions of $x^2+7^{2k+1}=y^n$}\label{sec:case_a}

In this section we determine the non-trivial primitive solutions of equation
\begin{equation}\label{eq:main_a}
x^2+7^{2k+1}=y^n,
\end{equation}
with $k\geq 0$, $n\geq 11$ an odd prime and $y$ even. Actually, we prove the following theorem.

\begin{theorem}\label{thm:case_a}
There are no non-trivial primitive solutions of \eqref{eq:main_a} with $y$ even and $n$ an odd prime with $n\geq 11$ in the following two cases
\begin{itemize}
\item $n\equiv 13,17,19,23\pmod{24}$ and $k\not\equiv 1\pmod 3$,
\item $n\equiv 5,7,13,23\pmod{24}$ and $k\equiv 1\pmod 3$.
\end{itemize}
\end{theorem}

Suppose $(a,b)$ is a primitive solution of \eqref{eq:main_a} with $2\mid b$. Because of Proposition \ref{prop:b_equal_2} we also assume that $b$ has an odd prime factor. For a solution $(a,b)$ it is very natural to attach a Frey curve of signature $(n,n,2)$ \cite{BennettSkinner04}. However, this Frey curve does not have enough local information to apply the symplectic argument. For this reason we also attach a Frey curve of signature $(2,3,n)$ \cite{DarmonGranville95}. The last Frey curve was studied in details in \cite{Barros10}. The local information of the $(2,3,n)$ Frey-curve is rich enough and together with the powerful multi-Frey technique \cite{BugeaudMignotteSiksek08a,BugeaudMignotteSiksek08b} allow us to give a contradiction.

\vspace*{0.3cm}
\subsection{\textbf{Frey curve of signature $(n,n,2)$:}} First of all, we can assume that $2k+1\not\equiv 0\pmod n$, otherwise the solution $(a,b)$ corresponds to a non-trivial primitive solution of the equation
\begin{equation}
x^2 = y^n + z^n,
\end{equation}
which does not have non--trivial primitive solutions \cite{DarmonMerel97}. 

Suppose $2k+1\not\equiv 0\pmod n$. Because $n\geq 11$ is an odd prime we have $\ord_2(b^n)\geq 11$. Based on the work of Bennett-Skinner \cite{BennettSkinner04} and assuming that $a\equiv 1\pmod 4$ we associate to the solution $(a,b)$ the Frey curve
\begin{equation}
E_1:~Y^2+XY=X^3 + \frac{a-1}{4}X^2+\frac{b^n}{64}X.
\end{equation}
The invariants of $E_1$ are the followings,
\begin{equation}
j(E_1)=2^6\frac{(-4a^2+3b^n)^3}{7^{2k+1}b^{2n}},\qquad \Delta(E_1) = -\frac{7^{2k+1}b^{2n}}{2^{12}},\qquad N(E_1)=7\prod_{p\mid b}p.
\end{equation}
Suppose $\rhoEo$ is the representation of $\Gal(\bar\QQ/\QQ)$ on the $n$--th torsion points $E_1[n]$ of $E_1$.

\begin{lemma}\label{lem:level_lowering_E1}
There exists a rational newform of level $N_1=14$ such that $\rhoEo\simeq\rhofo$.
\end{lemma}

\begin{proof}
The curve $E_1$ does not have complex multiplication and $\rhoEo$ is absolutely irreducible \cite[Corollary 3.1, Corollary 2.2]{BennettSkinner04}. By level lowering \cite{Ribet90} we know that there exists a newform $f_1$ of level $N_1 = 14$ such that $\rhoEo\simeq\rhofofn$. However, there is only one newform of level $N_1=14$ which is rational with corresponding elliptic curve
\begin{equation}
E_{f_1}:~ Y^2+YX+Y = X^3+4X-6.
\end{equation}
\end{proof}

\begin{lemma}\label{lem:3_nmid_a}
Suppose $(a,b)$ is a non-trivial primitive solution of \eqref{eq:main_a} with $n\geq 11$. Then $3\nmid a$.
\end{lemma}

\begin{proof}
Suppose $3\mid a$. Then we have that $\rhoEo\simeq\rhofo$ for the unique rational newform $f_1$ of level $N_1=14$. It holds $a_3(f_1)=-2$ while $a_3(E_1)=0$. According to Proposition \ref{prop:congruence_criterion} we understand that $n\mid 6$, which is a contradiction to the assumption $n\geq 11$.
\end{proof}

\vspace*{0.3cm}
\subsection{\textbf{Frey curve of signature $(2,3,n)$:}} At this point we introduce and study the arithmetic of the Frey curve of signature $(2,3,n)$ for a non-trivial primitive solution $(a,b)$ of \eqref{eq:main_a_lemma}. The arithmetic of this curve has been studied in details in \cite{Barros10}. As above we make the assumption that $a\equiv 1\pmod 4$. Because of Lemma \ref{lem:3_nmid_a} we have $3\nmid a$ and assume that $a\equiv 1\pmod 3$. Because $a,b$ are coprime we conclude that $3\nmid b$.

Suppose that $2k+1 = 3\plambda + i$ where $i = 0,1,2$. It holds that $\plambda$ and $i$ do not have the same parity. The associate Frey curve is given by
\begin{equation}
\pEo:~Y^2 = X^3+3\cdot7^{\plambda+i}X +(-1)^{i+1}2\cdot 7^i\cdot a.
\end{equation}
The curve is minimal away from $2$ and the minimal invariants (for $p=2$ this follows from Lemma \ref{lem:conductor_pEo}) of $\pEo$ are
\begin{align*}
j(\pEo) & = 2^6\cdot 3^3\cdot \frac{7^{2k+1}}{b^n},\qquad \Delta(\pEo) = -2^{-6}\cdot 3^3\cdot 7^{2i}\cdot b^n,\\
c_4(\pEo) & = -3^2\cdot 7^{\plambda + i}, \qquad c_6(\pEo) = (-1)^i\cdot 3^3\cdot 7^i\cdot a.
\end{align*}

\begin{lemma}\label{lem:conductor_pEo}
The conductor of $\pEo$ is given by
\begin{equation}
N(\pEo) = \begin{cases} 2\cdot 3^{\epsilon_3}\cdot 7^2\cdot \Rad_{2,3}(b), & i\neq 0, \\ 2\cdot 3^{\epsilon_3}\cdot \Rad_{2,3}(b), & i = 0,\end{cases}
\end{equation}
with $\epsilon_3 = 2,3$. Moreover, the reduction type at $p=3$ is II for $e_3=3$ and III for $e_3=2$.
\end{lemma}

\begin{proof}
For $p\geq 5$ this follows from the invariants of $\pEo$. For $p=3$ we carry out Tate's algorithm for $\pEo$ together with \cite[TABLEAU II]{Papadopoulos93} and we compute $\epsilon_3$ and the reduction type (see also \cite[Table 5.8]{Barros10}). 

The case $p=2$ is more complicated and we use the ideas of the proof in \cite[Lemma 4.6]{PoonenSchaeferStoll07}. We recall that we have assumed $a\equiv 1\pmod 4$ and $2|b$. We also assume that $i\equiv 1\pmod 2$. Because $2\nmid a$ there exists $t=a/7^{\plambda}\in\ZZ_2$. Because $a^2 + 7^{3\plambda + i}\equiv 0\pmod{2^7}$, this implies that $(a,7^{\plambda})\equiv(-t^3/7^i,-t^2/7^i)\pmod{2^7}$. The right-hand side of the Weierstrass model of $\pEo$ is congruent to $(X-2t)(X+t)^2\pmod{2^7}$, hence the substitution $X\mapsto X-t$ gives a cubic polynomial congruent to $X^3-3tX^2\pmod{2^7})$. Because $a\equiv 1\pmod 4$ and $i\equiv 1\pmod 2$ we have $t\equiv 1\pmod 4$. The substitution $Y\mapsto Y+X$ shows that the model is not minimal. Thus for the minimal model we have $v_2(c_{4,min})=v_2(c_4(\pEo)) - 4=0$, so $\pEo$ has multiplicative reduction at $2$. The case $i\equiv 0\pmod 2$ is similar.
\end{proof}

Because $b\neq 1$ we understand that $\pEo$ does not have complex multiplication. Suppose $\rhopEo$ is the representation of $\Gal(\bar\QQ/\QQ)$ on the $n$--th torsion points $\pEo[n]$ of $\pEo$.

\begin{lemma}\label{lem:irreducible_pEo}
The representation $\rhopEo$ is irreducible for $n\geq 11$.
\end{lemma}

\begin{proof}
Suppose $\rhopEo$ is reducible, then $\pEo$ has a $\QQ$-rational subgroup of order $n$. By the work of Mazur \cite{Mazur78} we have that for $n\geq 11$ and $n\neq 13$ this can not happen unless $b$ is a power of $2$. However, we have made the assumption that $b$ is not a power of $2$ because of Proposition \ref{prop:b_equal_2}.

Suppose $n=13$ and $\rhopEo$ is reducible. Then $\pEo$ corresponds to a rational point on the modular curve $X_0(13)$. The $j$-map from $X_0(13)$ to $X(1)$ is given by
\begin{equation}
j_{13}(t)=\frac{(t^4 + 7t^3 + 20t^2 + 19t + 1)^3(t^2 + 15t + 13)}{t}.
\end{equation}
Then we have $j(\pEo)=j_{13}(t)$ for some $t\in\QQ$. Let assume that $t=x/y$ with $x,y$ non-zero coprime integers. Then we have
\begin{equation}
2^6\cdot 3^3\cdot \frac{7^{2k+1}}{b^n} = \frac{(x^4 + 7x^3y + 20x^2y^2 + 19xy^3 + y^4)^3(x^2 + 15xy + 13y^2)}{y^{13}x}
\end{equation}
We can easily see that $(x^4 + 7x^3y + 20x^2y^2 + 19xy^3 + y^4)^3(x^2 + 15xy + 13y^2)$ and $y^{13}x$ are coprime, thus $2^6\cdot 3^3\cdot 7^{2k+1}=(x^4 + 7x^3y + 20x^2y^2 + 19xy^3 + y^4)^3(x^2 + 15xy + 13y^2)$. We define $h_1 = x^4 + 7x^3y + 20x^2y^2 + 19xy^3 + y^4$ and $h_2=x^2 + 15xy + 13y^2$. Using \texttt{Sage} we can prove that
\begin{align*}
\res(h_1,h_2;x) & = -2^2\cdot 3^2\cdot 11^2\cdot 23\cdot y^8,\\
\res(h_1,h_2;y) & = -2^2\cdot 3^2\cdot 11^2\cdot 23\cdot x^8.
\end{align*}
Hence, $h_1$ and $h_2$ are coprime away from $2,3$. Thus we conclude that one of the following cases holds,
\begin{equation}
\pm 1,\pm 2,\pm 3,\pm 4, \pm 6, \pm 12 = h_1(x,y) \qquad \text{or}\qquad \pm 2^{a_1}\cdot 3^{a_2} = h_2(x,y),
\end{equation}
with $0\leq a_1\leq 6$ and $0\leq a_2\leq 3$. For the initial case we solve Thue equations using \texttt{Pari} and we find the solutions for which $xy\neq 0$ to be $(x,y)\in \{(\pm 1,\pm 1),(\pm 2,\pm 1)\}$. However, for any of the above pairs $(x,y)$ it holds $v_7(j_{13}(x/y))=0$ which shows that it can not be equal to the $j(\pEo)$. 

For the latter case we understand that $v_7(h_1(x,y))>0$. But an elementary computation shows that $h_2(x,y)\not\equiv 0\pmod 7$ for coprime $x,y$. 
\end{proof}

We define
$$
\pNo = \begin{cases} 2\cdot 3^{\epsilon_3}\cdot 7^2, & i\neq 0, \\ 
2\cdot 3^{\epsilon_3}, & i = 0,
\end{cases}
$$
with $\epsilon_3 = 2,3$. 

\begin{lemma}\label{lem:level_lowering_pEo}
There exists a newform of level $\pNo$ such that $\rhopEo\simeq\rhopfofn$.
\end{lemma}

\begin{proof}
Because $b\neq 1$ we know that $\pEo$ does not have complex multiplication. From Lemma \ref{lem:irreducible_pEo} $\rhopEo$ is irreducible, hence by level lowering \cite{Ribet90} we know that there exists a newform $\pfo$ of level $\pNo$ such that $\rhopEo\simeq\rhopfofn$.
\end{proof}

There are cases in the modular method when Proposition \ref{prop:congruence_criterion} is not enough to eliminate newforms. A standard technique is to compare Galois representations locally, especially the order of the inertia group $I_\ell$ for a suitable choice of the prime $\ell$. The following lemma gives all the local information we need for $\rhopEo$.

\begin{lemma}\label{lem:inertia_pEo}
Suppose $I_\ell$ is the inertia subgroup of $\Gal(\overline{\QQ}/\QQ)$ with respect to the prime $\ell$. Then we have
\begin{itemize}
\item $$
\#\rhopEo(I_7) = \begin{cases} 1, & i = 0, \\ 6, & i = 1,\\ 3, & i = 2. \end{cases}
$$

\item $$\#\rhopEo(I_3)=\begin{cases} 4, & \epsilon_3 = 2,\\ 12, & \epsilon_3 = 3.\end{cases}$$
\end{itemize}
\end{lemma}

\begin{proof}
For $p=7$ this is an immediate consequence of the computations of the invariants of $\pEo$ and \cite[Proposition 1]{Kraus90}.

Suppose $p=3$. We recall that $3\nmid ab$, so $v_3(\Delta(\pEo))=3$ and $v_3(c_6(\pEo))=3$. From Lemma \ref{lem:conductor_pEo} we know that $\pEo$ has $III$ reduction type when $\epsilon_3=2$ and has $II$ reduction type when $\epsilon_3=3$. Combining this with the Corollary of \cite[Th\'{e}or\`{e}m 1]{Kraus90} on page $356$ we have the result.
\end{proof}

\vspace*{0.3cm}
\subsection{proof of Theorem \ref{thm:case_a}}

With the above notation for a non-trivial primitive solution $(a,b)$, the rational newform $f_1$ of level $14$ and prime $\ell\nmid 48$ we define
\begin{equation}
R_1(f_1,a)  = 
\begin{cases}
a_\ell(E_1) - a_\ell(f_1), & if~\left(\frac{-7}{\ell}\right)=-1,\\ 
(\ell+1)^2 - a_\ell(f_1)^2, & if~\left(\frac{-7}{\ell}\right)=1.
\end{cases}
\end{equation}
Similarly, for a newform $\pfo$ of level $\pNo$ we define
\begin{equation}
\pRo(\pfo,a)  = 
\begin{cases}
\norm_{K_{\pfo}/\QQ}(a_\ell(\pEo) - a_\ell(\pfo)), & if~\left(\frac{-7}{\ell}\right)=-1, \\ \norm_{K_{\pfo}/\QQ}((\ell+1)^2 - a_\ell(\pfo)^2), & if~\left(\frac{-7}{\ell}\right)=1.
\end{cases}
\end{equation}
Now let
$$
T_\ell(f_1,\pfo)=
\begin{cases}\dis\prod_{0\leq a\leq \ell-1}\gcd(R_1(f_1,a), \pRo(\pfo,a)), &\text{if }\pfo~\text{is rational}, \\
\dis\ell\cdot\prod_{0\leq a\leq \ell-1}\gcd(R_1(f_1,a), \pRo(\pfo,a)), &\text{otherwise}.
\end{cases}.
$$
Because of Proposition \ref{prop:congruence_criterion}, Lemma \ref{lem:level_lowering_E1} and Lemma \ref{lem:level_lowering_pEo} we have that $n\mid T_\ell(f_1,\pfo)$. 

We have written a simple \texttt{Sage} script that computes $U(f_1,\pfo) = \dis\gcd_{\ell\leq 31}(T_\ell(f_1,\pfo))$ for each possible pair $(f_1,\pfo)$. For the cases $\pEfo$ is rational we also check if $\#\rhopEo(I_\ell)$ and $\#\rhopfo(I_\ell)$ are equal for $\ell=3,7$ using Lemma \ref{lem:inertia_pEo}. For these pairs that $U(f_1,\pfo)\neq 0$ we get $n\leq 7$. 

However, there are six pairs $(f_1,\pfo)$ such that $U(f_1,\pfo)= 0$ and $\#\rhopEo(I_\ell)=\#\rhopfo(I_\ell)$ for $\ell=3,7$. This holds for six rational newforms $\pfo$ with corresponding elliptic curve $\pEfo$ as in the Table \ref{table:exc_pEfo}. We will eliminate these six exceptional cases using the symplectic argument. For a prime $\ell$ and elliptic curve $E$ we denote by $\tilde\Delta_\ell(E)$ (or $\tilde\Delta(E)$) the prime to $\ell$ part of $\Delta(E)$.

\begin{table}
\centering
\begin{tabular}{| c | c | c |}
\hline
Cremona label & $\pEfo:~(a_1,a_2,a_3,a_4,a_6)$ & $i$ \\
\hline
$54b1$ & $(1,-1,1,1,-1)$ & $0$\\
\hline
$882g1$ & $(1,-1,1,1,39)$ & $1$\\
\hline
$882f1$ & $(1,-1,1,64,-13597)$  & $2$\\
\hline
$2646bc1$ & $(1,-1,1,1,-3)$ & $1$\\
\hline
$2646q1$ & $(1,-1,1,64,809)$ & $2$\\
\hline
\end{tabular}
\caption{The exceptional curves $\pEfo$ given by Cremona label.}\label{table:exc_pEfo}
\end{table}

\begin{lemma}\label{lem:symplectic_l2}
Suppose $\pEfo$ is one curve from Table \ref{table:exc_pEfo}. We have that $\pEfo[n]$ and $\pEo[n]$ are sympletically isomorphic if and only if $(-6/n)=1$ for $i\neq 0$ and $(-2/n)=1$ for $i=0$.
\end{lemma}

\begin{proof}
We apply\cite[Theorem 13]{FreitasKraus17} for $\ell=2$. Both $\pEo$ and $\pEfo$ have multiplicative reduction at $2$ and $n\nmid v_2(\Delta(\pEo))$. It holds $v_2(\Delta(\pEfo))=1$ or $9$ for $i\neq 0$ and $v_2(\Delta(\pEfo))=3$ for $i=0$. In addition, we have $v_2(\Delta(\pEo))=-6+nv_2(b)$. Thus from \cite[Theorem 13]{FreitasKraus17} we have that $\pEfo[n]$ and $\pEo[n]$ are sympletically isomorphic if and only if $(-6/n)=1$ for $i\neq 0$ and $(-2/n)=1$ for $i=0$.
\end{proof}

\begin{lemma}\label{lem:symplectic_l3}
Suppose $\pEfo$ is one curve from Table \ref{table:exc_pEfo}. We have that $\pEfo[n]$ and $\pEo[n]$ are sympletically isomorphic.
\end{lemma}

\begin{proof}
We apply the symplectic criteria in \cite{FreitasKraus17} for $\ell=3$. We have to work according to the value of $\epsilon_3$.

\vspace*{0.2cm}
\noindent\textbf{Case $\epsilon_3=2$:} It holds $v_3(\Delta(\pEfo))\equiv 3\pmod 4$ and $\left(\frac{\tilde\Delta(\pEfo)}{3}\right)=1$. Moreove, we have $v_3(\Delta(\pEo))\equiv 3\pmod 4$ and
$\left(\frac{\tilde\Delta(\pEo)}{3}\right)=1$ because $3\nmid a$ (Lemma \ref{lem:3_nmid_a}). With the notation of \cite[Theorem 5]{FreitasKraus17} we have that $r=0$ and $t=0$. Thus, we conclude from \cite[Theorem 5]{FreitasKraus17} $\pEfo[n]$ and $\pEo[n]$ are sympletically isomorphic.

\vspace*{0.2cm}
\noindent\textbf{Case $\epsilon_3=3$:} According to Lemma \ref{lem:inertia_pEo} we have that $\#\rhopfo(I_3)=12$. Then, according to \cite[Theorem 10]{FreitasKraus17} $\pEfo[n]$ and $\pEo[n]$ are sympletically isomorphic whenever $(3/n)=1$.

Suppose $(3/n)=-1$. It holds $(v_3(c_4(\pEfo)),v_3(c_6(\pEfo)),v_3(\Delta(\pEfo)))=(2,3,3)$ while $(v_3(c_4(E_1)),v_3(c_6(E_1)),v_3(\Delta(E_1)))=(2,3\text{ or }4,3)$ because $3\nmid a$ (Lemma \ref{lem:3_nmid_a}). Then, according to \cite[Theorem 11]{FreitasKraus17} $\pEfo[n]$ and $\pEo[n]$ are sympletically isomorphic.
\end{proof}

Now, Theorem \ref{thm:case_a} is an immediate consequence of the above analysis and Lemmas \ref{lem:symplectic_l2} and \ref{lem:symplectic_l3}.

\section{Primitive solutions of $7x^2+1=y^n$}\label{sec:case_b}

In this section we determine the primitive solutions of equation
\begin{equation}\label{eq:main_b}
7x^2+1=y^n,
\end{equation}
with $n\geq 11$ an odd prime and $y$ even. Suppose $(a,b)$ is a primitive solution of \eqref{eq:main_a} with $2\mid b$. As in Section \ref{sec:case_a} we apply the modular method. Because $n\geq 11$ is an odd prime we have that $\ord_2(b^n)\geq 11$. Again, based on the work of Bennett-Skinner \cite{BennettSkinner04} and assuming that $a\equiv 3\pmod 4$ we associate to the solution $(a,b)$ the Frey curve,
\begin{equation}
E_2:~Y^2+XY=X^3 + \frac{7a-1}{4}X^2+\frac{7b^n}{64}X.
\end{equation}
The invariants of $E_2$ are the followings,
\begin{equation}
j(E_2)=-2^6\frac{(7a^2-3)^3}{b^{2n}},\qquad \Delta(E_2) = -\frac{7^3b^{2n}}{2^{12}},\qquad N(E_2)=7^2\prod_{p\mid b}p.
\end{equation}
Suppose $\rhoEt$ is the representation of $\Gal(\bar\QQ/\QQ)$ on the $n$--th torsion points $E_2[n]$ of $E_2$. The curve $E_2$ does not have complex multiplication and $\rhoEt$ is absolutely irreducible according to \cite[Corollary 3.1, Corollary 2.2]{BennettSkinner04}. By level lowering \cite{Ribet90} we know that there exists a newform $f_2$ of level $N_n(E_2)=98$ such that $\rhoEt\sim_n\rhoftfn$. 

\begin{theorem}\label{thm:case_b}
There are no non-trivial primitive solutions of \eqref{eq:main_b} with $y$ even and $n\geq 11$.
\end{theorem}

\begin{proof}
Suppose $(a,b)$ is a non-trivial primitive solution of \eqref{eq:main_b}. From the above there exists a
newform $f_2$ of level $N_n(E_2)=98$ such that $\rhoEt\sim_n\rhoftfn$. There are two newforms of level $98$, one is irrational and the other is rational.

Suppose $f_2$ is the irrational newform of level $98$. Then the eigenvalue field of $f_2$ is $K=\QQ(\theta)$ with $\theta^2-2\theta-1=0$. It holds $a_3(f_2)=\theta - 1$ while $a_3(E_2)=0,\pm 2$ (from Weil's bound). Due to Proposition \ref{prop:congruence_criterion} we conclude that $n\mid\norm_{K/\QQ}(a_3(f_2) - a_3(E_2))$ or $n\mid\norm_{K/\QQ}(a_3(f_2)^2 - 4^2)=14^2$. This implies that $n\mid 14$ which is a contradiction to the fact that $n\geq 11$.

Suppose $f_2$ is the rational newform of level $98$. A corresponding elliptic curve due to modularity is the following
\begin{equation}
E_{f_2}:~ Y^2+YX = X^3+X^2-25X-111.
\end{equation}
It is known that $\#\rhoEt(I_7)=4$ \cite{Kraus90,BennettSkinner04}. However, $E_{f_2}$ has potential multiplicative reduction because $v_7(j(E_{f_2}))=-1$. By the theory of Tate curve we have that $n\mid\#\rhoft(I_7)$ (see \cite[Chapter V]{Silvermanbook}) which gives a contradiction for $n\geq 11$.
\end{proof}

\section{Acknowledgement}

The author would like to thank Professor Michael Bennett for useful conversations and suggestions about a first draft. We are also grateful to Adela Gherga for solving all the Thue-Mahler equations we asked her. Finally, we want to thank Professor John Cremona for the access to the servers of Number Theory Group of University of Warwick where all the computations the author did took place.

\bibliographystyle{alpha}
\bibliography{on_the_equation_x2_7k_yn}

\end{document}